\documentclass[11pt,reqno]{amsart}

\usepackage{mathrsfs}
\usepackage{amssymb}
\usepackage{amsthm}
\usepackage{amsmath,amsfonts,amssymb,esint}
\usepackage{graphics,color}
\usepackage{enumerate}
\usepackage{marginnote}

\newtheorem{theorem}{Theorem}[section]
\newtheorem{lemma}[theorem]{Lemma}
\newtheorem{proposition}[theorem]{Proposition}
\newtheorem{corollary}[theorem]{Corollary}
\newtheorem{definition}[theorem]{Definition\rm}
\newtheorem{remark}{Remark}

\newcommand{\T}{\ensuremath{\mathbb{T}}}
\newcommand*{\R}{\ensuremath{\mathbb{R}}}
\renewcommand*{\S}{\ensuremath{\mathcal{S}}}
\newcommand*{\N}{\ensuremath{\mathbb{N}}}
\newcommand*{\Z}{\ensuremath{\mathbb{Z}}}
\newcommand*{\C}{\ensuremath{\mathbb{C}}}

\renewcommand*{\div}{\ensuremath{\mathrm{div\,}}}
\newcommand*{\tr}{\ensuremath{\mathrm{tr\,}}}

\newcommand*{\Id}{\ensuremath{\mathrm{Id}}}

\renewcommand*{\P}{\ensuremath{\mathcal{P}}}
\newcommand*{\Q}{\ensuremath{\mathcal{Q}}}
\newcommand*{\RR}{\ensuremath{\mathcal{R}}}

\begin{document}

\begin{abstract}
We show the existence of continuous periodic solutions of the 3D incompressible Euler
equations which dissipate the total kinetic energy.
\end{abstract}

\title[Dissipative continuous Euler Flows]
{Dissipative continuous Euler Flows}

\author{Camillo De Lellis}
\address{Institut f\"ur Mathematik, Universit\"at Z\"urich, CH-8057 Z\"urich}
\email{camillo.delellis@math.unizh.ch}

\author{L\'aszl\'o Sz\'ekelyhidi Jr.}
\address{Institut f\"ur Mathematik, Universit\"at Leipzig, D-04103 Leipzig}
\email{laszlo.szekelyhidi@math.uni-leipzig.de}

\maketitle

\section{Introduction}

In what follows $\T^3$ denotes the $3$-dimensional torus, i.e. $\T^3 = {\mathbb S}^1\times {\mathbb S}^1 \times
{\mathbb S}^1$. In this note we prove the following theorem.

\begin{theorem}\label{t:main}
Assume $e: [0, 1]\to \R$ is a positive smooth function.
Then there is a continuous vector field $v: \T^3 \times [0, 1]\to \R^3$ and a continuous
scalar field $p:\T^3\times [0,1]\to \R$ which solve the incompressible
Euler equations 
\begin{equation}\label{e:Euler}
\left\{\begin{array}{l}
\partial_t v + \div (v\otimes v) + \nabla p =0\\ \\
\div v = 0
\end{array}\right.
\end{equation}
in the sense of distributions and such that
\begin{equation}\label{e:energy_id}
e(t) = \int |v|^2 (x,t)\, dx\qquad\forall t\in [0, 1]\, .
\end{equation}
\end{theorem}

Obviously, if we choose a strictly decreasing function $e(t)$, Theorem \ref{t:main} yields
continuous solutions of the incompressible Euler equations which ``dissipate'' the total kinetic
energy $\frac{1}{2} \int |v|^2 (x,t)\, dx$. This is not possible for $C^1$ solutions: in that case
one can multiply the first equation in \eqref{e:Euler} by $v$ to derive  
\[
\partial_t \frac{|v|^2}{2} + {\rm div} \left(u \left( \frac{|u|^2}{2} + p\right)\right) \;=\; 0\, .
\]
Integrating this last identity in $x$ we then conclude
\begin{equation}\label{e:conservation}
\frac{d}{dt} \int_{\T^3} \frac{|v|^2}{2} (x,t)\, dx \;=\; 0\, .
\end{equation}
Theorem \ref{t:main} shows therefore that this formal computation cannot be justified for distributional
solutions, even if they are continuous.
The pair $(v,p)$ in Theorem \ref{t:main} solves \eqref{e:Euler} in the following sense:
\begin{equation}\label{e:test1}
\int_0^1\int_{\T^3}\left( \partial_t\varphi\cdot v+\nabla\varphi:v\otimes v + p\,  {\rm div}\, \varphi\right)\,dxdt=0
\end{equation}
for all $\varphi\in C_c^{\infty}(\T^3\times (0,1);\R^3)$
and
\[
\int_0^1\int_{\T^3}v\cdot \nabla\psi\,dxdt=0 \qquad \mbox{for all $\psi\in C_c^{\infty}(\T^3\times (0,1))$}.
\]
\begin{remark} 
In the usual definition of weak solution, \eqref{e:test1} is replaced by the same condition for divergence free
test fields: therefore $p$ disappears from the identity. With this alternative definition, for every weak solution $v$ which belongs to $L^2$ 
a corresponding pressure field can then be recovered using
\begin{equation}\label{e:pressure}
-\Delta p = {\rm div}\, {\rm div}\, (v\otimes v)\, .
\end{equation}
$p$ is then determined up to an arbitrary function of $t$: this arbitrariness can be overcome by imposing,
for instance, $\fint p (x,t)\, dx= 0$.
However, as it is well-known, the equation \eqref{e:pressure} and the continuity of $v$ does not guarantee the continuity of $p$.
\end{remark}

\subsection{Onsager's Conjecture} The possibility that weak solutions might dissipate
the total kinetic energy has been considered for a rather long time in the fluid dynamics literature: this
phenomenon goes under the name of ``anomalous dissipation''. In fact, to our knowledge, the existence
of dissipative solutions was considered for the first time by
Lars Onsager in his famous 1949 note about statistical hydrodynamics, see \cite{Onsager}. In that
paper Onsager conjectured that 
\begin{itemize}
\item[(a)] $C^{0,\alpha}$ solutions are energy conservative when $\alpha> \frac{1}{3}$;
\item[(b)] There exist dissipative solutions with $C^{0, \alpha}$ regularity for any $\alpha <\frac{1}{3}$
\end{itemize}
(note that, though Onsager's definition of ``weak solution'' is, strictly speaking, different from the one given
above, it can be easily shown that the two concepts are equivalent).

\medskip

The first part of the conjecture, i.e.~assertion (a), has been shown by Eyink in \cite{Eyink} and by Constantin, E and Titi
in \cite{ConstantinETiti}. The proof of the last paper amounts to give a rigorous justification of the
formal computation sketched above and leading to \eqref{e:conservation}: this is done via a suitable regularization of the equation and some careful commutator
estimates. The second part of the conjecture, i.e.~statement (b), is still widely open. A first result in that direction was
the groundbreaking work of Scheffer \cite{Scheffer93} which proved the existence of a compactly supported nontrivial
weak solution in $\R^2\times \R$. A different construction of the existence of a compactly supported nontrivial
weak solution in $\T^2\times \R$ was then given by Shnirelman in \cite{Shnirelman1}. In both cases the 
solutions are only square summable (as a function of both space and time variables): it is therefore not clear whether there are intervals of time in which
the total kinetic energy is a monotone function (indeed it is not even clear whether these solutions belong
to the energy space $L^\infty_t (L^2_x)$). The first proof of the existence of a solution for which the total
kinetic energy is a monotone decreasing function has been
given by Shnirelman in \cite{Shnirelmandecrease}. Shnirelman's example is only in the energy space
$L^\infty ([0, \infty[, L^2 (\R^3))$. 

\subsection{$h$-principle} Our work \cite{DS1,DS2} showed the existence of dissipative solutions for
which both pressure and velocity are bounded. 
Besides the obvious improvement (and the discovery of quite severe counterexamples to the uniqueness
of admissible solutions, both for incompressible and compressible Euler), in this work we introduced a new
point of view in the subject, highlighting connections to other counterintuitive solutions of 
(mainly geometric) systems of partial differential equations: in geometry these solutions are, according to Gromov, instances of the $h$-principle. 
In particular the Onsager's Conjecture bears
striking similarities with the rigidity and flexibility properties of isometric embeddings of 
Riemannian manifolds, pioneered by the celebrated work of Nash \cite{Nash54}. Indeed, results of the same
flavor as statements (a) and (b) can be proved in the case of isometric embeddings (see for instance \cite{CDSz} and the references therein): in comparing the Onsager's conjecture and these results,
the reader should take into account that, in this analogy, the velocity field of the Euler equations corresponds
to the differential of the embedding in the isometric embedding problem.
All these aspects (and further developments for some PDEs in fluid dynamics inspired by our work) are surveyed in the note \cite{DS3}. See also \cite{Chiodaroli, CFG, Shvydkoy, Szekelyhidi, SzWie, Wiedemann}.

\medskip

The understanding of Nash's construction was in a way a starting point for our approach to Theorem \ref{t:main}.
As in the case of Nash, the solution of \eqref{e:Euler} is generated by an iteration scheme: at each stage of this iteration we produce an ``almost solution'' which solves Euler with an additional error term. We name the resulting ``perturbed'' system of equations Euler-Reynolds system, since the error term has the typical form of the
so-called Reynolds stress in the fluid dynamics literature (see \cite{DS3} for an informal discussion of this point).
This error term
converges to $0$, while the sequence of almost solutions converge to an exact solution, uniformly in $C^0$. 
At each stage the new approximate solution is generated from the previous one by adding some special
perturbations, which oscillate quite fast. In our case the building blocks of these perturbations are Beltrami flows, a special
class of exact oscillatory solutions to the Euler equations. Hence, the final result of the iteration scheme is the superposition of infinitely many (perturbed) and weakly interacting Beltrami flows. Curiously, 
the idea that turbulent flows can be understood as a superposition of Beltrami flows has been proposed
almost 30 years ago in the fluid dynamics literature: see the work of Constantin and Majda \cite{ConstantinMajda}. Indeed,
it was Peter Constantin who suggested to us to try Beltrami flows in a convex integration scheme.

\subsection{Comments on the proof}

Though the proof of Theorem \ref{t:main} shares several similarities with Nash's scheme, 
there are many points where our method departs dramatically from Nash's, due to some
issues which are typical of the Euler equations and are not present for the isometric embeddings. 

{\bf 1)}\, Perhaps the most important new aspect of our scheme is a ``transport term'' which arises, roughly speaking,
as the linearization of the first equation in \eqref{e:Euler}: this term is typical of an evolution equation, whereas, instead, the equations for isometric embeddings are ``static''. At a first glance this transport term makes it impossible to use a scheme like the one of Nash to prove Theorem \ref{t:main}. To overcome this obstruction we need to introduce a phase-function that acts as a kind of discrete Galilean transformation of the (stationary) Beltrami flows, and to introduce an ``intermediate'' scale along each iteration step on which this transformation acts.  

{\bf 2)}\, The scheme introduced by Nash and its wide generalization by Gromov known as convex integration heavily relies on one-dimensional oscillations - the simple reason being that these can be ``integrated'', hence the name convex integration. As already mentioned, the main building blocks of our iteration scheme are Beltrami flows, which are truly three-dimensional oscillations. The issue of going beyond one-dimensional oscillations has been raised by Gromov (p219 of \cite{Gromov}) as well as Kirchheim-M\"uller-\v Sver\'ak (p52 of \cite{KMS02}), but as far as we know, there have been no such examples in the literature so far. In fact, it seems that with one-dimensional oscillations alone one cannot overcome the obstruction in 1).

{\bf 3)}\, A third, more technical, new aspect is the absence of a simple potential to generate solutions of the Euler-Reynolds system: in a sense we cannot simply ``integrate'' Beltrami flows. 
In order to overcome this issue we introduce a ``corrector term'' to the main perturbation. This corrector term is not ``explicit'': it is determined by solving some appropriate elliptic equations. 

{\bf 4)}\, In order to estimate the corrector and show that its contribution is negligible compared to that of the main perturbation, we use a combination of standard Schauder theory and oscillatory integrals estimates. This gives to our proof a ``hard'' PDE flavor compared to the construction of Nash,
which is more on the side of ``soft analysis''.

{\bf 5)}\, As a minor comment we remark that obviously the smoothness of $e$ in Theorem \ref{t:main} can be relaxed, but we do not pursue this issue here. 
Moreover, the same theorem can be proved if we
replace $[0,1]$ with $[0, \infty[$: in this case we require in addition that $e$ and its derivatives
are uniformly bounded and that there is a positive constant $c_0$ with $e\geq c_0$.

{\bf 6)}\, In Theorem \ref{t:main} our aim was to construct continuous weak solutions. In particular we did not address issues concerning the initial-value problem. However, this is a typical way to proceed with problems involving the $h$-principle. 
In general one may distinguish two aspects: a local and a global one. In geometric situations the local one is typically a differential constraint whereas the global one is topological (cf. \cite{Eliashberg}). The flexibility (in other words the lack of uniqueness) that one observes in instances of the $h$-principle is tied to the specifics of the local aspect. Thus, our Theorem \ref{t:main} deals exclusively with the ``local'' aspect for the Euler equations, whereas a possible analogue of the global aspect would be the imposition of an initial data, possibly together with an admissibility condition (as in \cite{DS2}). In subsequent papers we plan to address such ``global'' issues (e.g. initial data, compactly supported ancient solutions, etc).

\subsection{Acknowledgements}
We wish to thank Peter Constantin and Sergio Conti for several very valuable discussions on earlier attempts to prove Theorem \ref{t:main}. Moreover we are grateful to Antoine Choffrut for several comments on earlier versions of the paper, which considerably improved its readability. The first author acknowledges the support of the SFB Grant TR71, the second author acknowledges the support of the ERC Grant Agreement No.~277993 and the support of the Hausdorff Center for Mathematics in Bonn.
 
\section{Setup and plan of the paper}

The proof of Theorem \ref{t:main} will be achieved through an iteration procedure. Along the iteration the maps will be ``almost solutions'' of the Euler equations. To measure ``how far''
a solenoidal field is from being a solution of incompressible Euler we introduce a system of differential
equations which we call Euler-Reynolds system. The name is justified by the fact that the matrix-field
$\mathring{R}$ is a well known object in the theory of turbulence, called ``Reynolds stress'' (cp. with
\cite{DS3} and the references therein). In what follows $\mathcal{S}^{3\times 3}_0$ denotes the vector space
of symmetric trace-free $3\times 3$ matrices.

\begin{definition}\label{d:euler_reynolds}
Assume $v, p, \mathring{R}$ are smooth functions on $\T^3\times [0, 1]$
taking values, respectively, in $\R^3, \R, \S^{3\times 3}_0$. We say that they solve the Euler-Reynolds system if 
\begin{equation}\label{e:euler_reynolds}
\left\{\begin{array}{l}
\partial_t v + \div (v\otimes v) + \nabla p =\div\mathring{R}\\ \\
\div v = 0\, .
\end{array}\right.
\end{equation} 
\end{definition}

We are now ready to state the main proposition of this paper, of which Theorem \ref{t:main}
is a simple corollary.

\begin{proposition}\label{p:iterate}
Let $e$ be as in Theorem \ref{t:main}.
Then there are positive constants $\eta$ and $M$ with the following property.

Let $\delta\leq 1$ be any positive number and $(v, p, \mathring{R})$  a solution of the Euler-Reynolds
system \eqref{e:euler_reynolds} such that
\begin{equation}\label{e:energy_hyp}
\tfrac{3\delta}{4} e(t) \leq e(t) - \int |v|^2 (x,t)\, dx \leq \tfrac{5\delta}{4} e(t) \qquad \forall t\in [0,1]
\end{equation}
and
\begin{equation}\label{e:Reynolds_hyp}
\sup_{x,t}|\mathring{R}(x,t)| \leq \eta \delta\, .
\end{equation}
Then there is a second triple $(v_1, p_1, \mathring{R}_1)$ which solves
as well the Euler-Reynolds system and satisfies the following estimates:
\begin{equation}\label{e:energy_est}
\tfrac{3\delta}{8} e(t) \leq e(t) - \int |v_1|^2 (x,t)\, dx\leq \tfrac{5\delta}{8} e(t) \qquad \forall t\in [0,1]\, ,
 \end{equation}
\begin{equation}\label{e:Reynolds_est}
\sup_{x,t}|\mathring{R}_1(x,t)|\leq \tfrac{1}{2}\eta\delta\, ,
\end{equation}
\begin{equation}\label{e:C^0_est}
\sup_{x,t}|v_1(x,t) - v(x,t)| \leq M \sqrt{\delta}\, 
\end{equation}
and
\begin{equation}\label{e:pressure_est}
\sup_{x,t} |p_1 (x,t) - p (x,t)| \leq M \delta\, .
\end{equation}
\end{proposition}

As already mentioned, Theorem \ref{t:main} follows immediately from Proposition \ref{p:iterate}.

\begin{proof}[Proof of Theorem \ref{t:main}]
We start by setting $v_0=0$, $p_0=0$, $\mathring{R}_0=0$ and $\delta := 1$. We then
apply Proposition \ref{p:iterate} iteratively to reach a sequence $(v_n, p_n, \mathring{R}_n)$ which
solves \eqref{e:euler_reynolds} and such that
\begin{eqnarray}
\frac{3}{4}\frac{e(t)}{2^n} \leq e(t) - \int |v_n|^2 (x,t)\, dx&\leq& \frac{5}{4}\frac{e(t)}{2^n}\qquad \mbox{for all $t\in [0,1]$}\label{e:energy_final}\\
\sup_{x,t}|\mathring{R}_n(x,t)|&\leq& \frac{\eta}{2^n}\\
\sup_{x,t}|v_{n+1}(x,t) - v_n(x,t)| &\leq& M \sqrt{\frac{1}{2^n}}\\
\sup_{x,t} |p_{n+1} (x,t)-p_n (x,t)| &\leq& \frac{M}{2^n}\, .
\end{eqnarray}
Then $\{v_n\}$ and $\{p_n\}$ are both Cauchy sequences in $C(\T^3\times[0,1])$ and converge uniformly to two continuous
functions $v$ and $p$. Similarly $\mathring{R}_n$ converges uniformly to $0$. Moreover, by \eqref{e:energy_final}
\[
\int_{\T^3} |v|^2 (x,t)\, dx = e(t) \qquad \forall t\in [0,1]\, .
\]
Passing into the limit in \eqref{e:euler_reynolds} we therefore conclude that
$(v,p)$ solves \eqref{e:Euler}.
\end{proof}

\subsection{Construction of $v_1$} The rest of the paper will be dedicated to prove Proposition \ref{p:iterate}.
The construction of the map $v_1$ consists of adding two perturbations to $v$:
\begin{equation}\label{e:define_w}
v_1 = v + w_o + w_c =:v+w.
\end{equation}
To specify the form of the perturbation $w_o$, which is a highly oscillatory function 
and for which we give a rather explicit formula, we need several ingredients. 
The vectorfield $v+w_o$ is not in general divergence free. 
Therefore we add the correction $w_c$ to restore this condition. Having added the correction, the 
main focus will then be on finding maps $\mathring{R}_1$ and $p_1$ with the desired estimate and such that
\[
\partial_t v_1 + {\rm div}_x (v_1\otimes v_1) + \nabla p_1 = {\rm div}_x \mathring{R}_1\, . 
\]
The perturbation $w_o$ will depend on two parameters, $\mu$ and $\lambda$, which will
satisfy the following conditions
\begin{equation}\label{e:integrality}
\lambda, \mu, \frac{\lambda}{\mu} \in \N\, .
\end{equation}
In order to achieve the estimates, $\lambda$ and $\mu$ will be chosen quite large, depending
on appropriate norms of $v$. As already mentioned, the building blocks for the perturbation $w_o$ are
Beltrami flows. In order to give the formula leading to the definition of $w_o$ we must, therefore,
study closer the particular ``geometry'' of  these flows. This will be done in the next section.
We will then be ready to define the perturbations $w_o$ and $w_c$: this task will be accomplished
in Section \ref{s:perturbations} where we also prescribe the constants $\eta$ and $M$ of the
estimates in Proposition \ref{p:iterate}. After recalling some classical Schauder theory in Section \ref{s:schauder},
in the Sections \ref{s:estimates_energy} and \ref{s:reynolds} we will prove the relevant estimates of the various
terms involved in the construction, in terms of the parameters $\lambda$ and $\mu$. The choice of these parameters
will be finally specified 
in Section \ref{s:conclusion}, where we conclude the proof of Proposition \ref{p:iterate}. 

\section{Geometric preliminaries}

In this paper we denote by $\R^{n\times n}$, as usual, the space of $n\times n$ matrices, whereas $\mathcal{S}^{n\times n}$
and $\mathcal{S}^{n\times n}_0$ denote, respectively, the corresponding subspaces of symmetric matrices 
and of trace-free symmetric matrices. The $3\times 3$ identity matrix will be denoted with $\Id$. 
For definitiveness we will use the matrix operator norm $|R|:=\max_{|v|=1}|Rv|$. Since we will
deal with symmetric matrices, we have the identity $|R|= \max_{|v|=1} |Rv \cdot v|$.

\subsection{Beltrami flows}
We start by recalling a celebrated example of stationary periodic solutions to the 3D Euler equations,
the so called Beltrami flows.
One important fact which will play a central role in our paper is that the space of Beltrami flows
contains linear spaces of fairly large dimension.

\begin{proposition}[Beltrami flows]\label{p:Beltrami}
Let $\lambda_0\geq 1$ and let $A_k\in\R^3$ be such that 
$$
A_k\cdot k=0,\,|A_k|=\tfrac{1}{\sqrt{2}},\,A_{-k}=A_k
$$
for $k\in\Z^3$ with $|k|=\lambda_0$.
Furthermore, let 
$$
B_k=A_k+i\frac{k}{|k|}\times A_k\in\C^3.
$$
For any choice of $a_k\in\C$ with $\overline{a_k} = a_{-k}$ the vectorfield
\begin{equation}\label{e:Beltrami}
W(\xi)=\sum_{|k|=\lambda_0}a_kB_ke^{ik\cdot \xi}
\end{equation}
is divergence-free and satisfies
\begin{equation}\label{e:Bequation}
\div (W\otimes W)=\nabla\frac{|W|^2}{2}.
\end{equation}
Furthermore
\begin{equation}\label{e:av_of_Bel}
\langle W\otimes W\rangle= \fint_{\T^3} W\otimes W\,d\xi = \frac{1}{2} \sum_{|k|=\lambda_0} |a_k|^2 \left( {\rm Id} - \frac{k}{|k|}\otimes\frac{k}{|k|}\right)\, .  
\end{equation}
\end{proposition}

In other words $W(\xi)$ defined by \eqref{e:Beltrami} is a stationary solution of \eqref{e:Euler} with pressure $p=-\frac{|W|^2}{2}$.
For the rest of this paper we will treat the vectors $A_k\in\R^3$, $B_k\in\C^3$ as fixed (the choice of $A_k$ as prescribed in the Proposition is not unique, but this is immaterial for our purposes). The proof of Proposition \ref{p:Beltrami}
is a classic in the fluid dynamics literature, but we include it for the reader's convenience.

\begin{proof} First of all observe that $a_{-k}B_{-k} = \overline{a_kB_k}$. Thus the vector field defined in \eqref{e:Beltrami}
is real valued. Next notice that 
\[
{\rm div} W (\xi)= \sum_{|k|=\lambda_0}  ik\cdot B_k a_k e^{ik\cdot \xi} = 0\, ,
\]
because $k\cdot B_k =0$ for every $k$. 

Observe also that 
\[
{\rm curl}\, W (\xi) = \sum_{|k|=\lambda_0} i k \times B_k a_k e^{ik\cdot \xi}\, .
\]
On the other hand
\begin{eqnarray*}
i k\times B_k &=& \lambda_0 \left( i \frac{k}{|k|}\times A_k - \frac{k}{|k|} \times \left(\frac{k}{|k|}\times A_k\right)\right)\\
&=& \lambda_0  \left( i \frac{k}{|k|}\times A_k + A_k\right) = \lambda_0 B_k\, .
\end{eqnarray*}
We therefore infer ${\rm curl}\, W = \lambda_0 W$. 
Since $W$ is divergence free, ${\rm div}\, (W\otimes W) = (W\cdot \nabla) W$ and we can use the well known
vector identity
\[
{\rm div}\, (W\otimes W) = (W\cdot\nabla) W = \nabla \frac{|W|^2}{2} - W\times ({\rm curl}\, W)\, .
\]
Since we have just seen that ${\rm curl}\, W$ and $W$ are parallel, \eqref{e:Bequation} follows easily.

Finally, we compute
\[
W\otimes W = \sum_{k,j} a_k a_j B_k\otimes B_j e^{i(k+j)\cdot \xi} =
\sum_{k,j} a_k\overline{a_j}B_k\otimes\overline{B}_j e^{i(k-j)\cdot \xi}\, .
\]
Averaging this identity in $\xi$ we infer
\[
\langle W\otimes W\rangle = \sum_{|k|=\lambda_0} |a_k|^2 B_k\otimes \overline{B}_k\, .
\]
However, since $B_k = \overline{B}_{-k}$, we get
\begin{eqnarray*}
\langle W\otimes W\rangle &=& \sum_{|k|=\lambda_0} |a_k|^2\, {\rm Re} \left(B_k\otimes \overline{B}_k\right)\\
&=& \sum_{|k|=\lambda_0} |a_k|^2 \left(A_k\otimes A_k + \left(\frac{k}{|k|}\times A_k\right)
\otimes \left(\frac{k}{|k|}\times A_k\right)\right)\, .
\end{eqnarray*}
On the other hand, observe that the triple $\sqrt{2} A_k, \sqrt{2} \frac{k}{|k|}\times A_k, \frac{k}{|k|}$
forms an orthonormal base of $\R^3$. Thus,
\[
2 A_k\otimes A_k + 2 \left(\frac{k}{|k|}\times A_k\right)
\otimes \left(\frac{k}{|k|}\times A_k\right) + \frac{k}{|k|}\otimes \frac{k}{|k|} = {\rm Id}\, .
\]
This shows \eqref{e:av_of_Bel} and hence completes the proof.
\end{proof}

\subsection{The geometric lemma} One key point of
our construction is that the abundance of Beltrami flows allows to find several such flows $v$ with the property
that
\[
\langle v\otimes v\rangle (t) := \frac{1}{(2\pi)^3} \int_{\T^3} v\otimes v (x,t) \, dx
\]
equals a prescribed symmetric matrix $R$. Indeed we will need to select these flows so as to depend
smoothly on the matrix $R$, at least when $R$ belongs to a neighborhood of the identity matrix.
In view of \eqref{e:av_of_Bel}, such selection is made possible by the following Lemma.

\begin{lemma}[Geometric Lemma]\label{l:split}
For every $N\in\N$ we can choose $r_0>0$ and $\lambda_0>1$ with the following property.
There exist pairwise disjoint subsets 
$$
\Lambda_j\subset\{k\in \Z^3:\,|k|=\lambda_0\} \qquad j\in \{1, \ldots, N\}
$$
and smooth positive functions 
\[
\gamma^{(j)}_k\in C^{\infty}\left(B_{r_0} (\Id)\right) \qquad j\in \{1,\dots, N\}, k\in\Lambda_j
\]
such that
\begin{itemize}
\item[(a)] $k\in \Lambda_j$ implies $-k\in \Lambda_j$ and $\gamma^{(j)}_k = \gamma^{(j)}_{-k}$;
\item[(b)] For each $R\in B_{r_0} (\Id)$ we have the identity
\begin{equation}\label{e:split}
R = \frac{1}{2} \sum_{k\in\Lambda_j} \left(\gamma^{(j)}_k(R)\right)^2 \left({\rm Id} - \frac{k}{|k|}\otimes \frac{k}{|k|}\right) 
\qquad \forall R\in B_{r_0}(\Id)\, .
\end{equation}
\end{itemize}
\end{lemma}

\begin{remark}
Though it will not be used in the sequel, the cardinality of each set $\Lambda_j$ constructed in the proof
of the Lemma is indeed bounded \emph{a priori} independently of all the other parameters. A close inspection of the
proof shows that it gives sets with cardinality at most 98. This seems however far from optimal: one should
be able to find sets $\Lambda_j$ with cardinality 14.
\end{remark}

The proof of the Geometric Lemma is based on the following well-known fact.

\begin{proposition}\label{p:Q_dense_in_S}
The set ${\mathbb Q}^3\cap {\mathbb S}^2$ is dense in ${\mathbb S}^2$.
\end{proposition}
\begin{proof}
Let ${\mathbf s}: {\mathbb R}^2\to {\mathbb S}^2$ be the inverse of the stereographic projection:
\[
{\mathbf s} (u,v) := \left(\frac{2v}{u^2+v^2+1}, \frac{2u}{u^2+v^2+1}, \frac{u^2+v^2 -1}{u^2+v^2+1}\right)\, .
\]
It is obvious that ${\mathbf s} ({\mathbb Q}^2)\subset {\mathbb Q}^3$. Since
${\mathbb Q}^2$ is dense in ${\mathbb R}^2$ and ${\mathbf s}$ is a diffeomorphism onto ${\mathbb S}^2\setminus (0,0,1)$,
the proposition follows trivially.
\end{proof}

Indeed, much more can be proved: $\frac{1}{n} {\mathbb Z}^3 \cap {\mathbb S}^2$,
distributes uniformly on the sphere for $n \in {\mathbb N}$ large whenever $n\equiv 1,2,3,4,5,6\; (\textrm{mod}\; 8)$.
This problem was raised by Linnik (see \cite{Linnik}), who proved a first result in its direction, and solved thanks to a 
breakthrough of Iwaniec \cite{Iwaniec} in the theory of modular forms of half-integral weight 
(see, for instance, \cite{Duke} and \cite{Sarnak}).

\begin{proof}[Proof of Lemma \ref{l:split}] For each vector $v\in \R^3\setminus \{0\}$, we denote by 
$M_v$ the $3\times 3$ symmetric matrix
given by 
\[
M_v = {\rm Id} - \frac{v}{|v|}\otimes \frac{v}{|v|}\, .
\] 
With this notation the identity \eqref{e:split} reads as
\begin{equation}\label{e:split2}
R = \frac{1}{2} \sum_{k\in\Lambda_j} \left(\gamma^{(j)}_k(R)\right)^2 M_k\, .
\end{equation}

\medskip

{\bf Step 1} Fix a $\lambda_0>1$ and for each set $F\subset \{k\in \Z^3: |k|= \lambda_0\}$ we consider 
the set $c (F)$ which is the interior of the convex hull, in $\S^{3\times 3}$, of $\{M_k:k\in F\}$.
We claim in this step that it suffices to find a $\lambda_0$ and $N$ disjoint subsets 
$F_j \subset \{k\in \Z^3: |k|= \lambda_0\}$ such that
\begin{itemize}
\item[(d)] $-F_j=F_j$;
\item[(e)] $c(F_j)$ contains a positive multiple of the identity.
\end{itemize}
Indeed, we will show below that, if $F_j$ satisfies (d) and (e), then we can find
a $r_0>0$, a subset $\Gamma_j\subset F_j$ and positive smooth functions
$\lambda^{(j)}_k\in C^\infty (B_{2r_0} (\Id))$ such that
\[
R = \sum_{k\in \Gamma_j} \lambda^{(j)}_k (R) M_k\, .
\]
We then find $\Lambda_j$ and the functions $\gamma^{(j)}_k$ by
\begin{itemize}
\item defining $\Lambda_j := \Gamma_j \cup - \Gamma_j$;
\item setting $\lambda^{(j)}_k = 0$ if $k \in \Lambda_j\setminus \Gamma_j$;
\item defining
\[
\gamma^{(j)}_k := \sqrt{\lambda_k^{(j)} + \lambda_{-k}^{(j)}}\, 
\]
for every $k\in \Lambda_j$.
\end{itemize}
Observe that the functions and the sets satisfy both (a) and (b). Moreover, since
at least one of the $\lambda^{(j)}_{\pm k}$ is positive on $B_{2r_0} (\Id)$, $\gamma^{(j)}_k$
is smooth in $B_{r_0} (\Id)$.

\medskip

We now come to the existence of the set $\Gamma_j$. For simplicity we drop the subscripts.
The open set $c(F)$ contains an element $\alpha {\rm Id}$ with $\alpha>0$. 
Then there are seven matrices $A_1, \ldots, A_7$ 
in $c(F)$ such that $\alpha {\rm Id}$ belongs to the interior of their convex hull, which is an open convex
simplex $S$. We choose $\vartheta$ so that the unit ball $\tilde{U}$ of center ${\alpha Id}$
and radius $\vartheta$ is contained in $S$. Then each point $R\in \tilde{U}$ can be written in a unique
way as a convex combination of the elements $A_i$: 
\[
R = \sum_{i=1}^7 \beta_i (R) A_i
\]
and the functions $\beta_i$ are positive and smooth on $\tilde{U}$. 

By Caratheodory's Theorem, each $A_i$ is the convex combination 
$\sum \lambda_{i,n} M_{v_{i,n}}$ of at most $7$ $M_{v_{i,n}}$ with $v_{i,n}\in F$,
where we require that each $\lambda_{i,n}$ is positive (observe that Caratheodory's Theorem
guarantees the existence of 7 points $M_{v_{i,n}}$ such that $A_i$ belongs to the {\em closed} 
convex hull of them; if we insist on the property that the corresponding coefficients are all
positive, then we might be obliged to choose a number smaller than $7$). 

Set $r_0:=\tfrac{\vartheta}{2\alpha}$. Then, 
\[
R = \sum_{i, n} \frac{1}{\alpha} \beta_i (\alpha R) \lambda_{i,n} M_{v_{i,n}}
\qquad \forall R\in B_{2r_0} (\Id)\, . 
\]
and each coefficient 
\[
\frac{1}{\alpha} \beta_i (\alpha R) \lambda_{i,n}
\]
is positive for every $R\in B_{2r_0} (\Id)$.

The set $\Gamma_j$ is then given by $\{v_{i,n}\}$. Note that we might have 
$v_{i,n}= v_{l,m}$ for two distinct pairs $(i,n)$ and $(l,m)$. Therefore,
for $k\in \Gamma_j$, the function $\lambda_k$ will be defined as
\[
\lambda_k (R) = \sum_{(i,n): k= v_{i,n}} \frac{1}{\alpha} \beta_i (\alpha R) \lambda_{i,n}\, .
\]

\medskip

{\bf Step 2} By Step 1, in order to prove the lemma, it suffices to find
a number $\lambda_0$ and $N$ disjoint families $F_1, \ldots, F_N \subset \lambda_0 {\mathbb S}^2\cap \Z^3$ 
such that the sets $c(F_i)$ contain all a positive multiple of the identity.
By Proposition \ref{p:Q_dense_in_S} there is a sequence $\lambda_k\uparrow \infty$  
such that the sets ${\mathbb S}^2\cap \frac{1}{\lambda_k} \Z^3$ converge, in the Hausdorff
sense, to the entire sphere ${\mathbb S}^2$. 

Given this sequence $\{\lambda_k\}$, we can easily partition each 
$\lambda_k {\mathbb S}^2\cap \Z^3$ into N disjoint symmetric families $\{F^k_j\}_{j=1, \ldots, N}$ in such a way that,
for each fixed $j$, the corresponding sequence of sets $\{\frac{1}{\lambda_k} F^k_j\}_k$ converges in the Hausdorff sense 
to ${\mathbb S}^2$. Hence, any point of $c ({\mathbb S}^2)$ is contained in $c (\frac{1}{\lambda_k} F^k_j)$ provided $k$ is large enough.
On the other hand it is easy to see that $c({\mathbb S}^2)$ contains a multiple of the identity $\alpha {\rm Id}$ 
(for instance one can adapt the argument of Lemma 4.2 in \cite{DS1}).
By Step 1, this concludes the proof.
\end{proof}

%%%%%%%%%%%%%%%%
\section{The maps $v_1$, $\mathring{R}_1$ and $p_1$}\label{s:perturbations}
%%%%%%%%%%%%%%%%

We have now all the tools to define the maps $v_1$, $\mathring{R}_1$ and $p_1$ of Proposition
\ref{p:iterate}. Recalling \eqref{e:define_w} and the discussion therein, $w:=v_1-v$ is the
sum of two maps, $w_o$ and $w_c$. $w_o$ is a highly oscillatory function based on ``patching
Beltrami flows'' and it will be defined first, in Section \ref{ss:w_o}. $w_c$ will then be added so as to ensure that 
$v_1$ is divergence free: in order to achieve this we will use the classical Leray projector, see
Section \ref{ss:w_c} for the precise definition. $p_1$ is related to $w_o$ by a simple formula, given
in Section \ref{ss:p_1}. Finally, in Section \ref{ss:R_1} we will define $\mathring{R}_1$. Essentially,
this last matrix field can also be thought of as a ``corrector term'', analogous to $w_c$. In fact, 
if we consider the point of view of \cite{DS1}, the Euler-Reynolds system can be stated equivalently as
the fact that the $4\times 4$ matrix 
\[
U:= \left(
\begin{array}{ll}
v_1\otimes v_1 + p_1 {\rm Id} - \mathring{R}_1 & v_1\\
v_1 & 0
\end{array}
\right)
\]
is a divergence-free in space-time. $\mathring{R}_1$ has therefore the same flavor
as $w_c$ and is also defined through a suitable (elliptic) operator, cp. with Definition \ref{d:reyn_op}.

\subsection{The perturbation $w_o$}\label{ss:w_o}
We start by defining a partition of unity on the space of velocities, i.e. the state space.
Choose two constants $c_1$ and $c_2$ such that $\frac{\sqrt{3}}{2} < c_1 < c_2 < 1$ and
we fix a function $\varphi\in C^\infty_c (B_{c _2} (0))$ which is nonnegative and identically
$1$ on the ball $B_{c_1} (0)$. 
We next consider the lattice $\Z^3\subset \R^3$ and its quotient by $(2\Z)^3$, i.e. we define the equivalence
relation
\[
(k_1, k_2, k_3) \sim (\ell_1, \ell_2, \ell_3) \quad \iff\quad \mbox{$k_i - \ell_i$ is even $\forall i$.}
\]
We then denote by $\mathcal{C}_j$ , $j=1, \ldots, 8$ the 8 equivalence classes of $\Z^3/\sim$.
For each $k\in \Z^3$ denote by $\varphi_k$ the function
\[
\varphi_k (x):= \varphi (x-k)\, .
\]
Observe that, if $k\neq \ell \in \mathcal{C}_i$, then $|k-\ell|\geq 2 > 2 c_2$. Hence 
$\varphi_k$ and  $\varphi_\ell$ have disjoint supports. On the other hand, the function
\[
\psi:=\sum_{k\in \Z^3} \varphi_k^2
\]
 is smooth, bounded and bounded away from zero. We then define
\[
\alpha_k(v) := \frac{\varphi_k(v)}{\sqrt{\psi(v)}}
\]
and 
\begin{equation*}
\phi^{(j)}_k(v,\tau) := \sum_{l\in\mathcal{C}_j}\alpha_l(\mu v)e^{-i(k\cdot \frac{l}{\mu})\tau}.
\end{equation*}
Since $\alpha_l$ and $\alpha_{\tilde l}$ have disjoint supports for $l\neq \tilde l\in\mathcal{C}_j$, 
it follows that for all $v,\tau,j$
\begin{equation}\label{e:phisum}
|\phi^{(j)}_k(v,\tau)|^2=\sum_{l\in\mathcal{C}_j}\alpha_l(\mu v)^2,
\end{equation}
and in particular $\sum_{j=1}^8|\phi^{(j)}_k(v,\tau)|^2=1$. Furthermore, for the same reason
there exist for any $m=0,1,2,\dots$ constants $C=C(m)$ such that
\begin{equation}\label{e:phiestimate1}
\sup_{v,\tau}|D^m_v\phi^{(j)}_k(v,\tau)|\leq C(m)\mu^m.
\end{equation}
Fix next any $(v, \tau)$ and $j$. Observe that there is at most one $l\in \mathcal{C}_j$
with the property that $\alpha_l (\mu v)\neq 0$ and this $l$ has the property that
$|\mu v-l|<1$. Thus, in a neighborhood of $(v, \tau)$ we will have
\begin{equation}\label{e:nearest_l}
\partial_{\tau}\phi^{(j)}_k+i(k\cdot v)\phi^{(j)}_k=ik\cdot \left(v-\frac{l}{\mu}\right)\phi^{(j)}_k,
\end{equation}
Combining \eqref{e:phiestimate1} and \eqref{e:nearest_l}, for any $m=0,1,2,\dots$ we find
constants $C=C(m,|k|)$ such that
\begin{equation}\label{e:phiestimate2}
\sup_{v,\tau}|D^m_v(\partial_{\tau}\phi^{(j)}_k+i(k\cdot v)\phi^{(j)}_k|\leq C(m,|k|)\mu^{m-1}.
\end{equation}

\bigskip

We apply Lemma \ref{l:split} with $N=8$ to obtain $\lambda_0>1$, $r_0>0$ and pairwise disjoint families $\Lambda_j$ together with corresponding
functions $\gamma^{(j)}_k\in C^{\infty}\left(B_{r_0}(\Id)\right)$. 

\bigskip

Next, set 
\[
\rho (t):= \frac{1}{3 (2\pi)^3} \left(e (t) \left(1-\frac{\delta}{2}\right) - \int_{\T^3} |v|^2 (x,t)\, dx\right)\, 
\]
and
\[
R (x,t):= \rho (t) \Id - \mathring{R} (x,t)\,,
\]
and define
\begin{equation}\label{e:w_o}
w_o (x,t) := \sqrt{\rho(t)}\sum_{j=1}^8\sum_{k\in\Lambda_j}\gamma_k^{(j)}\left(\frac{R(x,t)}{\rho(t)}\right)\phi_k^{(j)}\left(v(x,t),\lambda t\right)B_ke^{i\lambda k\cdot x}.
\end{equation}

\subsection{The constants $\eta$ and $M$}
Note that $w_o$ is well-defined only if $\frac{R}{\rho}\in B_{r_0}(\Id)$ where $r_0$ is given in Lemma \ref{l:split}. This is ensured by an appropriate choice of $\eta$. Indeed,
\[
\rho (t)\geq \frac{1}{3(2\pi)^3}\frac{\delta}{4}e(t) \geq c\delta \min_{t\in [0,1]} e(t) =: c \delta m\, ,
\]
where $c$ is a dimensional (positive) constant and $m>0$ by assumption.
Then
\[
\left\|\frac{R}{\rho (t)} - \Id\right\|\leq \frac{1}{c \delta m} \|\mathring{R}\|\leq \frac{\eta}{c m}\, .
\]
Thus, it suffices to choose 
\begin{equation}\label{e:choiceeta}
\eta := \tfrac{1}{2}c m r_0 = \frac{r_0}{24 (2\pi)^3} \min_{t\in [0,1]} e(t)\, .
\end{equation}
Observe that this choice
is independent of $\delta>0$. 

Notice next that, by our choice of $\rho(t)$ and by \eqref{e:energy_hyp}, $\rho (t)\leq \delta e(t)$.  
Thus there exists a constant $M>1$ depending only on $e$ (in particular independent of $\delta$) so that
\begin{equation}\label{e:est_wo}
\|w_o\|_0\leq \frac{\sqrt{M\delta}}{2}\, .
\end{equation}
This fixes the choice of the constant $M$ in Proposition \ref{p:iterate}

\subsection{The correction $w_c$}\label{ss:w_c}

We next define the Leray projector onto divergence-free vectorfields with zero average.

\begin{definition}\label{d:leray}
Let $v\in C^\infty (\T^3, \R^3)$ be a smooth vector field. Let
\begin{equation*}
\Q v:=\nabla\phi+\fint_{\T^3} v,
\end{equation*}
where $\phi\in C^{\infty}(\T^3)$ is the solution of
$$
\Delta\phi=\div v \textrm{ in }\T^3
$$
with $\fint_{\T^3}\phi=0$. Furthermore, let $\P=I-\Q$ be the Leray projection onto divergence-free fields with zero average.
\end{definition}

The vector field $v_1$ is then the sum of $v$ with the Leray projection $w$ of $w_o$, namely
\begin{equation*}
v_1 (x,t) := v (x,t) + \P w_o (x,t) =: v (x,t) + w(x,t)
\end{equation*}
and hence
\begin{equation*}
w_c (x,t) := - \Q w_o (x,t) = w (x,t) - w_o (x,t). 
\end{equation*}

\subsection{The pressure $p_1$}\label{ss:p_1}
We define
\begin{equation}\label{e:p_1}
p_1 := p - \frac{|w_o|^2}{2}\, . 
\end{equation}

\subsection{The Reynolds stress $\mathring{R}_1$}\label{ss:R_1}
In order to specify the choice of $\mathring{R}_1$ we introduce a new operator.

\begin{definition}\label{d:reyn_op}
Let $v\in C^\infty (\T^3, \R^3)$ be a smooth vector field. 
We then define $\RR v$ to be the matrix-valued periodic function
\begin{equation*}
\RR v:=\frac{1}{4}\left(\nabla\P u+(\nabla\P u)^T\right)+\frac{3}{4}\left(\nabla u+(\nabla u)^T\right)-\frac{1}{2}(\div u){\rm Id},
\end{equation*}
where $u\in C^{\infty}(\T^3,\R^3)$ is the solution of
\begin{equation*}
\Delta u=v-\fint_{\T^3}v\textrm{ in }\T^3
\end{equation*}
with $\fint_{\T^3} u=0$.
\end{definition}

\begin{lemma}[$\RR=\textrm{div}^{-1}$]\label{l:reyn}
For any $v\in C^\infty (\T^3, \R^3)$ we have
\begin{itemize}
\item[(a)] $\RR v(x)$ is a symmetric trace-free matrix for each $x\in \T^3$;
\item[(b)] $\div \RR v = v-\fint_{\T^3}v$.
\end{itemize}
\end{lemma}
\begin{proof}
It is obvious by inspection that $\RR v$ is symmetric. Since $\P v$ is divergence-free, we obtain for the trace
$$
\textrm{tr}(\RR v)=\frac{3}{4}(2\div u)-\frac{3}{2}\div u=0.
$$
Similarly, we have
\begin{equation}\label{e:div_of_Rv}
\div(\RR v)=\frac{1}{4}\Delta(\P u)+\frac{3}{4}(\nabla\div u+\Delta u)-\frac{1}{2}\nabla\div u.
\end{equation}
On the other hand recall that $\P u=u-\nabla\phi-\fint u=u-\nabla \phi$, where $\Delta\phi=\div u$. Therefore
$\Delta(\P u)=\Delta u-\nabla\div u$. Plugging this identity into \eqref{e:div_of_Rv},
we obtain
\[
\div(\RR v) = \Delta u\, 
\]
and since $u$ solves $\Delta u = v - \fint v$, (b) follows readily.
\end{proof}

Then we set
\begin{equation*}
\mathring{R}_1 :=\RR\left( \partial_t v_1  + \div(v_1\otimes v_1)+ \nabla p_1\right) .
\end{equation*}
Note that $[\partial_t v_1  + \div(v_1\otimes v_1)+ \nabla p_1]$ has average zero. Indeed:
\begin{itemize}
\item the vector field $\div(v_1\otimes v_1)+ \nabla p_1$ has average $0$ because it is the divergence of the matrix field $v_1\otimes v_1 + p_1{\rm Id}$;
\item for the same reason, the identity $\partial_t v  = - \div(v\otimes v + p{\rm Id}\, - \mathring{R})$ shows
that $\partial_t v$ has average 
$0$; on the other hand $w= \P w_o$ has average $0$ because of the definition of
$\P$; this implies that $\partial_t w$ has also average zero and thus we conclude as well that
$\partial_t v_1 = \partial_t v + \partial_t w$ has average zero.
\end{itemize}
Therefore from Lemma \ref{l:reyn} it follows that $\mathring{R}_1 (x,t)$ is symmetric and trace-free and that the
identity
\[
 \partial_t v_1 + \div (v_1\otimes v_1)
+ \nabla p_1 = \div\, \mathring{R}_1
\]
holds.

The rest of this note is devoted to prove that the triple $(v_1, p_1, \mathring{R}_1)$
satisfies the estimates \eqref{e:energy_est}, \eqref{e:Reynolds_est},  \eqref{e:C^0_est} and \eqref{e:pressure_est}.
This will be achieved by an appropriate choice of the parameters $\mu$ and $\lambda$.
In particular, we will show that the estimates hold provided $\mu$ is sufficiently large and
$\lambda$ much larger than $\mu$.

%%%%%%%%%%%%%%%%%%%%%%%%%
\section{Schauder Estimates}\label{s:schauder}
%%%%%%%%%%%%%%%%%%%%%%%%%

In the following $m=0,1,2,\dots$, $\alpha\in (0,1)$, and $\beta$ is a multiindex. We introduce the usual (spatial) 
H\"older norms as follows.
First of all, the supremum norm is denoted by $\|f\|_0:=\sup_{\T^3}|f|$. We define the H\"older seminorms 
as
\begin{equation*}
\begin{split}
[f]_{m}&=\max_{|\beta|=m}\|D^{\beta}f\|_0\, ,\\
[f]_{m+\alpha} &= \max_{|\beta|=m}\sup_{x\neq y}\frac{|D^{\beta}f(x)-D^{\beta}f(y)|}{|x-y|^{\alpha}}\, .
\end{split}
\end{equation*}
The H\"older norms are then given by
\begin{eqnarray*}
\|f\|_{m}&=&\sum_{j=0}^m[f]_j\\
\|f\|_{m+\alpha}&=&\|f\|_m+[f]_{m+\alpha}.
\end{eqnarray*}

Recall the following elementary inequalities:
\begin{equation}\label{e:Holderinterpolation}
[f]_{s}\leq C\bigl(\varepsilon^{r-s}[f]_{r}+\varepsilon^{-s}\|f\|_0\bigr)
\end{equation}
for $r\geq s\geq 0$, and 
\begin{equation}\label{e:Holderproduct}
[fg]_{r}\leq C\bigl([f]_r\|g\|_0+\|f\|_0[g]_r\bigr)
\end{equation}
for any $1\geq r\geq 0$.

Finally, we recall the classical Schauder estimates for the Laplace operator and the corresponding
estimates which we can infer for the various operators involved in our construction.

\begin{proposition}\label{p:GT}
For any $\alpha\in (0,1)$ and any $m\in \N$ there exists a constant $C (\alpha, m)$ with the following properties.
If $\phi, \psi: \T^3\to \R$ are the unique solutions of
\[
\left\{\begin{array}{l}
\Delta \phi = f\\ \\
\fint \phi =0
\end{array}\right.
\qquad\qquad 
\left\{\begin{array}{l}
\Delta \psi = {\rm div}\, F\\ \\
\fint \psi =0
\end{array}\right. \, ,
\]
then
\begin{equation}\label{e:GT_laplace}
\|\phi\|_{m+2+\alpha} \leq C (m, \alpha) \|f\|_{m, \alpha}
\quad\mbox{and}\quad \|\psi\|_{m+1+\alpha} \leq C (m, \alpha) \|F\|_{m, \alpha}\, .
\end{equation}
Moreover we have the estimates
\begin{eqnarray}
&&\|\mathcal{Q} v\|_{m+\alpha} \leq C (m,\alpha) \|v\|_{m+\alpha}\label{e:Schauder_Q}\\
&&\|\mathcal{P} v\|_{m+\alpha} \leq C (m,\alpha) \|v\|_{m+\alpha}\label{e:Schauder_P}\\
&&\|\mathcal{R} v\|_{m+1+\alpha} \leq C (m,\alpha) \|v\|_{m+\alpha}\label{e:Schauder_R}\\
&&\|\mathcal{R} ({\rm div}\, A)\|_{m+\alpha}\leq C(m,\alpha) \|A\|_{m+\alpha}\label{e:Schauder_Rdiv}\\
&&\|\mathcal{R} \mathcal{Q} ({\rm div}\, A)\|_{m+\alpha}\leq C(m,\alpha) \|A\|_{m+\alpha}\label{e:Schauder_RQdiv}\, .
\end{eqnarray}
\end{proposition}
\begin{proof} The estimates \eqref{e:GT_laplace} are the usual Schauder estimates,
see for instance \cite[Chapter 4]{GT}. The meticulous reader will notice 
that the estimates in \cite{GT} are stated in $\mathbb R^n$ for the potential-theoretic solution of the
Laplace operator. The periodic case is however an easy corollary. Take for instance $\phi$ and $f$ and consider them
as periodic functions defined on $\mathbb R^3$. Consider $g= f \chi$, where $\chi$ is a cut-off function supported in $B_{6\pi} (0)$
and identically $1$ on $B_{4\pi} (0)$. Let $\tilde{\phi}$ be the potential-theoretic solution in $\mathbb R^3$ of
$\Delta \tilde{\phi} = g$. For $\tilde{\phi}$ we can invoke the Schauder estimates as in \cite[Chapter 4]{GT}. Moreover
$\phi-\tilde{\phi}$ is an harmonic function in $B_{4pi} (0)$. Obviously $\|\phi\|_{L^2 (B_{4\pi} (0))}$ can be easily bounded using 
$\Delta \phi = f$, $\fint \phi =0$ and the Parseval identity. Thus, standard properties
of harmonic functions give $\|\phi-\tilde{\phi}\|_{C^{m,\alpha} ([2\pi]^3)}\leq C(m,\alpha)\|f\|_0$.

The estimates \eqref{e:Schauder_Q}, \eqref{e:Schauder_P}, \eqref{e:Schauder_R} and
\eqref{e:Schauder_Rdiv} are easy consequences of \eqref{e:GT_laplace} and the definitions of the operators. 
The estimate
\eqref{e:Schauder_RQdiv} requires a little more care. Let $u: \T^3\to \R^3$ be the unique solution of
\[
\Delta\Delta u_i = \partial_i \sum_{j,n} \partial^2_{jn} A_{jn}
\]
with $\fint u =0$. Then 
\begin{equation}\label{e:Schauder_biLaplace}
\|u\|_{m+1+\alpha}\leq C(m,\alpha) \|A\|_{m+\alpha}\, .
\end{equation} 
First of all, with the argument above, one can reduce this estimate to a corresponding one
for the potential-theoretic solution of the biLaplace operator in $\mathbb R^3$. For this case
we can then invoke general estimates for elliptic $k$-homogeneous constant coefficients operators (see for instance \cite[Theorem 1]{Simon_Schauder}) or use the same arguments of \cite[Chapter 4]{GT} replacing the fundamental solution of the Laplacian with that of the biLaplacian.
Finally, \eqref{e:Schauder_RQdiv} follows from the identity 
\[
\RR \mathcal{Q} ({\rm div}\, A)=\frac{1}{4}\left(\nabla\P u+(\nabla\P u)^T\right)+\frac{3}{4}\left(\nabla u+(\nabla u)^T\right)-\frac{1}{2}(\div u){\rm Id}\,  
\]
and the estimates \eqref{e:Schauder_biLaplace} and \eqref{e:Schauder_P}.
\end{proof}

In what follows we will use the convention that greek subscripts of H\"older norms denote always exponents
in the open interval $(0,1)$.

\bigskip

\begin{proposition}\label{p:schauder}
Let $k\in\Z^3\setminus\{0\}$ and $\lambda\geq 1$ be fixed. 

(i) For any $a\in C^{\infty}(\T^3)$ and $m\in\N$ we have
\begin{equation}\label{e:average}
\left|\int_{\T^3}a(x)e^{i\lambda k\cdot x}\,dx\right|\leq \frac{[a]_m}{\lambda^m}.
\end{equation}

(ii) Let $\phi_{\lambda}\in C^{\infty}(\T^3)$ be the solution
of
\begin{equation*}
\Delta\phi_{\lambda}=f_{\lambda}\textrm{ in }\T^3
\end{equation*}
with $\int_{\T^3}\phi_{\lambda}=0$,
where 
$$
f_{\lambda}(x):=a(x)e^{i\lambda k\cdot x}-\fint_{\T^3} a(y)e^{i\lambda k\cdot y}\,dy.
$$
Then for any $\alpha\in(0,1)$ and $m\in\N$ we have the estimate
\begin{equation}\label{e:schauder}
\|\nabla\phi_{\lambda}\|_{\alpha}\leq \frac{C}{\lambda^{1-\alpha}}\|a\|_0+\frac{C}{\lambda^{m-\alpha}}[a]_m+\frac{C}{\lambda^m}[a]_{m+\alpha},
\end{equation}
where $C=C(\alpha,m)$.
\end{proposition}

\begin{proof}
For $j=0,1,\dots$ define
\begin{align*}
A_j(y,\xi)&:=-i\left[\frac{k}{|k|^2}\left(i\frac{k}{|k|^2}\cdot\nabla\right)^ja(y)\right]e^{ik\cdot \xi}\, ,\\
F_j(y,\xi)&:=\left[\left(i\frac{k}{|k|^2}\cdot\nabla\right)^ja(y)\right]e^{ik\cdot \xi}\, .
\end{align*}
Direct calculation shows that
$$
F_j(x,\lambda x)=\frac{1}{\lambda}\div\bigl[A_j(x,\lambda x)\bigr]+\frac{1}{\lambda}F_{j+1}(x,\lambda x).
$$
In particular for any $m\in\N$
\begin{equation*}%\label{e:FdivA}
a(x)e^{i\lambda k\cdot x}=F_0(x,\lambda x)=\frac{1}{\lambda}\sum_{j=0}^{m-1} \frac{1}{\lambda^j}\div\bigl[A_j(x,\lambda x)\bigr]+\frac{1}{\lambda^{m}}F_m(x,\lambda x)
\end{equation*}
Integrating this over $\T^3$ and using that $|k|\geq 1$ we obtain \eqref{e:average}.

Next, using \eqref{e:Holderinterpolation} and \eqref{e:Holderproduct} we have for any $j\leq m-1$
\begin{equation*}
\begin{split}
\|A_j(\cdot,\lambda\cdot)\|_{\alpha}&\leq C\left(\lambda^{\alpha}[a]_j+[a]_{j+\alpha}\right)\\
&\leq C\lambda^{j+\alpha}\left(\lambda^{-m}[a]_m+\|a\|_0\right)
\end{split}
\end{equation*}
and similarly
\begin{equation*}
\|F_m(\cdot,\lambda\cdot)\|_{\alpha}\leq C\left(\lambda^{\alpha}[a]_m+[a]_{m+\alpha}\right)\, .
\end{equation*}
Moreover, according to standard \eqref{e:GT_laplace},
\begin{equation*}
\|\nabla\phi\|_{\alpha}\leq C\biggl(\frac{1}{\lambda}\sum_{j=0}^{m-1} \frac{1}{\lambda^j}\|A_j(\cdot,\lambda \cdot)\|_{\alpha}+\frac{1}{\lambda^{m}}\|F_m(\cdot,\lambda \cdot)\|_{\alpha}+\left|\fint_{\T^3} F_0(x,\lambda x)\,dx\right|\biggr),
\end{equation*}
hence, using \eqref{e:average} for the last term,
\begin{equation*}
\|\nabla\phi\|_{\alpha}\leq \frac{C}{\lambda^{1-\alpha}}\|a\|_0+\frac{C}{\lambda^{m-\alpha}}[a]_m+\frac{C}{\lambda^m}[a]_{m+\alpha}
\end{equation*}
as required.
\end{proof}

\bigskip

\begin{corollary}\label{c:schauder}
Let $k\in\Z^3\setminus\{0\}$ be fixed. For a smooth vectorfield $a\in C^{\infty}(\T^3;\R^3)$ let 
$F(x):=a(x)e^{i\lambda k\cdot x}$. Then we have
$$
\|\RR(F)\|_{\alpha}\leq \frac{C}{\lambda^{1-\alpha}}\|a\|_0+\frac{C}{\lambda^{m-\alpha}}[a]_m+\frac{C}{\lambda^m}[a]_{m+\alpha},
$$
where $C=C(\alpha,m)$.
\end{corollary}
\begin{proof}
This is an immediate consequence of the definition of $\RR$, the Schauder estimate \eqref{e:Schauder_P} for $\P$ and Proposition \ref{p:schauder} above.
\end{proof}

%%%%%%%%%%%%%%%%%%%%%%
\section{Estimates on the corrector and the energy}\label{s:estimates_energy}
%%%%%%%%%%%%%%%%%%%%%%

In all subsequent estimates, unless otherwise stated, $C$ denotes a generic constant that can vary from line to line, and 
depends on $e,v,\mathring{R}$ as well as on $\lambda_0,\alpha$ and $\delta$, but is independent of $\lambda$ and $\mu$. Smallness
of the respective quantities will be achieved by an appropriate choice of $\lambda,\mu$ in Section \ref{s:conclusion}.
Moreover, all estimates will implicitly assume \eqref{e:integrality}, in particular that $1\leq\mu\leq\lambda$. 

Our aim is to estimate the space-time sup-norm of $v_1-v=w = w_o+w_c$ and $\mathring{R}_1$. Since $w_c$ and $\mathring{R}_1$ are
defined in terms of the singular integral operators $\P,\Q$ and $\RR$, which act in space, instead of obtaining directly estimates of the $C^0$ norm, we will use Schauder estimates to obtain bounds on spatial H\"older norms. Thus, in the sequel the H\"older norms will denote spatial norms, and are understood to be uniform in time $t\in[0,1]$. Moreover, if the H\"older exponent is denoted by a greek letter, then it is a number in the open
interval $(0,1)$.

\bigskip

It will be convenient to write $w_o$ as
$$
w_o(x,t)=W(x,t,\lambda t,\lambda x),
$$
where
\begin{align}
W(y,s,\tau,\xi)&:=\sum_{|k|=\lambda_0}a_k(y,s,\tau)B_ke^{ik\cdot \xi}\label{e:bigW}\\
&=\sqrt{\rho(s)}\sum_{j=1}^8\sum_{k\in\Lambda_j}\gamma_k^{(j)}\left(\frac{R(y,s)}{\rho(s)}\right)\phi_k^{(j)}\left(v(y,s),\tau\right)B_ke^{ik\cdot \xi}.
\end{align}
We summarize the main properties of
the coefficients $W$:

\begin{proposition}\label{p:W}
(i) Let $a_k\in C^{\infty}(\T^3\times[0,1]\times\R)$ be given by \eqref{e:bigW}. Then for any $r\geq 0$ 
\begin{eqnarray*}
\|a_k(\cdot,s,\tau)\|_r&\leq & C\mu^r,\\
\|\partial_s a_k(\cdot,s,\tau)\|_r&\leq & C\mu^{r+1},\\
\|\partial_\tau a_k(\cdot,s,\tau)\|_r&\leq & C\mu^r,\\
\|(\partial_\tau a_k+i(k\cdot v)a_k)(\cdot,s,\tau)\|_r&\leq & C\mu^{r-1}.
\end{eqnarray*}

(ii) The matrix-function $W\otimes W$ can be written as
\begin{equation}\label{e:WoW}
(W\otimes W)(y,s,\tau,\xi)=R(y,s)+\sum_{1\leq |k| \leq 2\lambda_0}U_k(y,s,\tau)e^{ik\cdot \xi},
\end{equation}
where the coefficients $U_k\in C^{\infty}(\T^3\times[0,1]\times\R;\S^{3\times 3})$ satisfy
\begin{equation}\label{e:Uk}
U_kk=\frac{1}{2}(\tr U_k)k
\end{equation}
and for any $r\geq 0$
\begin{eqnarray*}
\|U_k^\mu(\cdot,s,\tau)\|_r&\leq & C\mu^r,\\
\|\partial_s U_k^\mu(\cdot,s,\tau)\|_r&\leq & C\mu^{r+1},\\
\|\partial_\tau U_k(\cdot,s,\tau)\|_r&\leq & C\mu^r,\\
\|(\partial_\tau U_k^\mu+i(k\cdot v)U_k^\mu)(\cdot,s,\tau)\|_r&\leq & C\mu^{r-1},\\
\end{eqnarray*}
In all these estimates the constant $C$ depends on $r$ and $e,v,\mathring{R}$ but is independent of $(s,\tau)$ and $\mu$.
\end{proposition}

\begin{proof}
The estimates for $a_k$ are a consequence of \eqref{e:phiestimate1} and \eqref{e:phiestimate2}. Indeed,
since the constants in the estimates are allowed to depend on $e,v,\mathring{R}$, one only needs to keep track of the number of derivatives
of $\phi_k^{(j)}$ with respect to $v$. Then the estimates on $a_k, \partial_sa_k$ and $\partial_{\tau}a_k+i(k\cdot v)a_k$ immediately follow.
From the triangle inequality we can then also conclude the estimate on $\partial_\tau a_k$.

Next, consider the expansion of $\xi\mapsto W\otimes W$ into a Fourier series in $\xi$, i.e.
$$
(W\otimes W)(y,s,\tau,\xi)=U_0(y,s,\tau)+\sum_{1\leq |k| \leq 2\lambda_0}U_k(y,s,\tau)e^{ik\cdot \xi}.
$$
Since each $U_k$ is the sum of finitely many terms of the form $a_{k'}a_{k''}$, the estimates for $U_k$ follow from those for $a_k$.

Next, since $U_0$ is given by the average (in $\xi$), 
in order to obtain \eqref{e:WoW} we need to show that
\begin{equation*}%\label{e:Waverage}
\fint_{\T^3}W\otimes W(y,s,\tau,\xi)\,d\xi=R(y,s).
\end{equation*}
To this end we calculate:
\begin{equation*}
\begin{split}
\fint_{\T^3}W\otimes W\,d\xi &\overset{\textrm{\eqref{e:av_of_Bel}}}{=} \frac{\rho}{2}\sum_j\sum_{k\in\Lambda_j}\left(\gamma_k^{(j)}(\rho^{-1}R)\right)^2|\phi^{(j)}_k(v,\tau)|^2\left(\textrm{Id}-\frac{k}{|k|}\otimes\frac{k}{|k|}\right)\\
&\overset{\textrm{\eqref{e:phisum}}}{=}\frac{\rho}{2}\sum_j\sum_{k\in\Lambda_j}\sum_{l\in\mathcal{C}_j}\left(\gamma_k^{(j)}(\rho^{-1}R)\right)^2\left(\alpha_l(v)\right)^2\left(\textrm{Id}-\frac{k}{|k|}\otimes\frac{k}{|k|}\right)\\
&\overset{\textrm{\eqref{e:split}}}{=}R\sum_j\sum_{l\in\mathcal{C}_j}\left(\alpha_l(v)\right)^2\\
&=R\sum_{l\in\Z^3}\left(\alpha_l(v)\right)^2=R.
\end{split}
\end{equation*}
Finally, \eqref{e:Uk} is a direct consequence of Proposition \ref{p:Beltrami}, in particular \eqref{e:Bequation}.
\end{proof}

After this preparation we are ready to estimate all the terms in the perturbation scheme. First of all we verify that
the corrector term $w_c$ is indeed much smaller than the main perturbation term $w_o$:

\begin{lemma}[Estimate on the corrector]
\begin{equation}
\|w_c\|_{\alpha}\leq C\frac{\mu}{\lambda^{1-\alpha}}\label{e:w_c_est}
\end{equation}
\end{lemma}

\begin{proof}
We start with the observation that, since $k\cdot B_k=0$, 
\begin{equation*}
\begin{split}
w_o(x,t)=&\frac{1}{\lambda}\nabla\times\left(\sum_{|k|=\lambda_0}-ia_k(x,t,\lambda t)\frac{k\times B_k}{|k|^2}e^{i\lambda x\cdot k}\right)+\\
&+\frac{1}{\lambda}\sum_{|k|=\lambda_0}i\nabla a_k(x,t,\lambda t)\times \frac{k\times B_k}{|k|^2}e^{i\lambda x\cdot k}.
\end{split}
\end{equation*}
Hence
\begin{equation}\label{e:w_cu_c}
w_c(x,t)=\frac{1}{\lambda}\Q u_c(x,t),
\end{equation}
where 
\begin{equation*}
u_c(x,t)=\sum_{|k|=\lambda_0}i\nabla a_k(x,t,\lambda t)\times \frac{k\times B_k}{|k|^2}e^{i\lambda x\cdot k}.
\end{equation*}

The estimate \eqref{e:w_c_est} then follows from the Schauder estimate \eqref{e:Schauder_Q} for $\Q$ combined with
$$
\|u_c\|_{\alpha}\leq C\mu\lambda^{\alpha}.
$$
\end{proof}

Next, we verify the estimate on the energy \eqref{e:energy_est}.

\begin{lemma}[Estimate on the energy]
\begin{equation}\label{e:energy}
\left|e(t)(1-\tfrac{1}{2}\delta)-\int_{\T^3}|v_1|^2\,dx\right|\leq C\frac{\mu}{\lambda^{1-\alpha}}.
\end{equation}
\end{lemma}

\begin{proof}
Taking the trace of identity \eqref{e:WoW} in Proposition \ref{p:W} we have
$$
|W(y,s,\tau,\xi)|^2=\tr R(y,s)+\sum_{1\leq |k| \leq 2\lambda_0}c_k(y,s,\tau)e^{ik\cdot \xi}
$$
for some coefficients $c_k\in C^{\infty}(\T^3\times[0,1]\times\R)$, which satisfy the estimates
$$
\|c_k(\cdot,s,\tau)\|_r\leq C\mu^{r}.
$$
From part (i) of Proposition \ref{p:schauder} with $m=1$ we deduce
\begin{equation*}
\left|\int_{\T^3}|w_o|^2-\tr R\,dx\right|\leq C\frac{\mu}{\lambda}
\end{equation*}
and 
\begin{equation*}
\left|\int_{\T^3}v\cdot w_o\,dx\right|\leq C\frac{\mu}{\lambda}.
\end{equation*}
Hence, combining with \eqref{e:w_c_est} we see that
\begin{equation*}
\left|\int_{\T^3}|v_1|^2-|v|^2-|w_o|^2\,dx\right|\leq C\frac{\mu}{\lambda^{1-\alpha}}\, .
\end{equation*}
Recalling that 
$$
\tr R=3\rho=\frac{1}{(2\pi)^3}\left(e(t)(1-\tfrac{1}{2}\delta)-\int_{\T^3}|v|^2\,dx\right),
$$
we conclude \eqref{e:energy}.
\end{proof}

\section{Estimates on the Reynolds stress}\label{s:reynolds}

Rewrite
\begin{equation}\label{e:reynolds-parts}
\begin{split}
\partial_tv_1+&\div(v_1\otimes v_1)+\nabla p_1=\left[\partial_tw_o+v\cdot \nabla w_o\right]+\\
&+\left[\div(w_o\otimes w_o-\tfrac{1}{2}|w_o|^2\textrm{Id}+\mathring{R})\right]\\
&+\left[\partial_t w_c+\div(v_1\otimes w_c+w_c\otimes v_1-w_c\otimes w_c+v\otimes w_o)\right]
\end{split}
\end{equation}
In other words we split the Reynolds stress into the three parts on the right hand side. We will refer to them as
the {\it transport part}, the {\it oscillation part}, and the {\it error}. In the following we will estimate each term separately.

\bigskip

\begin{lemma}[The transport part]
\begin{equation}\label{e:transport}
\|\RR(\partial_tw_o+v\cdot \nabla w_o)\|_{\alpha}\leq C\left(\frac{\lambda^{\alpha}}{\mu}+\frac{\mu^2}{\lambda^{1-\alpha}}\right)
\end{equation}
\end{lemma}

\begin{proof}
Observe that 
\begin{equation*}
\begin{split}
\RR(\partial_tw_o+v\cdot \nabla w_o)&=\lambda\RR\left(\sum_{|k|=\lambda_0}(\partial_\tau a_k+i(k\cdot v)a_k)(x,t,\lambda t)B_ke^{i\lambda k\cdot x}\right)+\\
&+\RR\left(\sum_{|k|=\lambda_0}(\partial_sa_k+v\cdot \nabla_ya_k)(x,t,\lambda t)B_ke^{i\lambda k\cdot x}\right).
\end{split}
\end{equation*}
For the first term Corollary \ref{c:schauder} with $m=2$ implies the bound
$$
\frac{\lambda^{\alpha}}{\mu}+\frac{\mu}{\lambda^{1-\alpha}}+\frac{\mu^{1+\alpha}}{\lambda},
$$
whereas for the second term Corollary \ref{c:schauder} with $m=1$ implies the bound
$$
\frac{\mu}{\lambda^{1-\alpha}}+\frac{\mu^2}{\lambda^{1-\alpha}}+\frac{\mu^{2+\alpha}}{\lambda}.
$$
Since $1\leq \mu\leq\lambda$, we obtain \eqref{e:transport}.
\end{proof}

\bigskip

\begin{lemma}[The oscillation part]
\begin{equation}\label{e:oscillation}
\|\RR(\div(w_o\otimes w_o-\tfrac{1}{2}|w_o|^2\textrm{Id}+\mathring{R})\|_{\alpha}\leq C\frac{\mu^2}{\lambda^{1-\alpha}}
\end{equation}
\end{lemma}

\begin{proof}
Recall the formula \eqref{e:WoW} from Proposition \ref{p:W}. 
Since $\rho$ is a function of $t$ only, we can write the oscillation part in \eqref{e:reynolds-parts} as
\begin{equation*}
\begin{split}
\div(w_o\otimes w_o-&\tfrac{1}{2}(|w_o|^2-\rho)\textrm{Id}+\mathring{R})\\
&=\div\left(w_o\otimes w_o-R-\tfrac{1}{2}(|w_o|^2-\tr R)\textrm{Id}\right)\\
&=\div\left[\sum_{1\leq |k| \leq 2\lambda_0}(U_k-\tfrac{1}{2}(\tr U_k)\textrm{Id})(x,t,\lambda t)e^{i\lambda k\cdot x}\right]\\
&=\sum_{1\leq |k| \leq 2\lambda_0}\div_y[U_k-\tfrac{1}{2}(\tr U_k)\textrm{Id}] (x, t, \lambda t) e^{i\lambda k\cdot x}
\end{split}
\end{equation*}
Corollary \ref{c:schauder} with $m=1$ then implies \eqref{e:oscillation}.
\end{proof}

\bigskip

Concerning the error, we are going to treat three terms separately, as follows.

\begin{lemma}[Estimate on the error I]
\begin{equation}\label{e:error1}
\|\RR(\partial_tw_c)\|_{\alpha}\leq C\frac{\mu^2}{\lambda^{1-\alpha}}
\end{equation}
\end{lemma}

\begin{proof}
Recall from \eqref{e:w_cu_c} that $w_c=\tfrac{1}{\lambda}\Q u_c$. Now
\begin{equation*}
\begin{split}
\partial_tu_c(x,t)=&\lambda \sum_{|k|=\lambda_0}i(\nabla \partial_\tau a_k)(x,t,\lambda t)\times \frac{k\times B_k}{|k|^2}e^{i\lambda x\cdot k}+\\
&\,+\sum_{|k|=\lambda_0}i(\nabla \partial_s a_k)(x,t,\lambda t)\times \frac{k\times B_k}{|k|^2}e^{i\lambda x\cdot k}\, .
\end{split}
\end{equation*}
Moreover, for any $c_k\in C^{\infty}(\T^3;\R^3)$ we have
\begin{equation*}
\begin{split}
c_k(x,t,\lambda t)e^{i\lambda x\cdot k}=\frac{1}{i\lambda}&\div\left[c_k(x,t,\lambda t)\otimes\frac{k}{|k|^2}e^{i\lambda x\cdot k}\right]-\\
&-\frac{1}{i\lambda}\left(\frac{k}{|k|^2}\cdot \nabla\right)c_k(x,t,\lambda t) e^{i\lambda x\cdot k}
\end{split}
\end{equation*}
Therefore $\partial_tu_c$ can be written as
$$
\partial_tu_c=\div U_c+\tilde u_c,
$$
where
$$
\|U_c\|_{\alpha}\leq C\mu\lambda^\alpha,\quad\|\tilde u_c\|_\alpha\leq C\mu^{2}\lambda^\alpha.
$$
Therefore we have $\RR \partial_tw_c=\frac{1}{\lambda}\left(\RR\Q\div U_c+\RR\Q\tilde u_c\right)$.
From the Schauder estimate \eqref{e:Schauder_RQdiv} for the operator $\RR\Q\div$, we conclude
that
\begin{equation*}
\begin{split}
\|\RR \partial_tw_c\|_\alpha&\leq \frac{1}{\lambda}\left(\|\RR\Q\div U_c\|_\alpha+\|\RR\Q\tilde u_c\|_\alpha\right)\\
&\leq \frac{C}{\lambda}\left(\|U_c\|_{\alpha}+\|\tilde u_c\|_{\alpha}\right)\leq C\frac{\mu^2}{\lambda^{1-\alpha}}
\end{split}
\end{equation*}
\end{proof}

\bigskip

\begin{lemma}[Estimate on the error II]
\begin{equation}\label{e:error2}
\|\RR\left(\div(v_1\otimes w_c+w_c\otimes v_1-w_c\otimes w_c)\right)\|_\alpha \leq C\frac{\mu}{\lambda^{1-2\alpha}}\, .
\end{equation}
\end{lemma}

\begin{proof}
We first estimate
\begin{equation*}
\begin{split}
\|v_1\otimes w_c+&w_c\otimes v_1-w_c\otimes w_c\|_\alpha\leq \\
&\leq C(\|v_1\|_0\|w_c\|_\alpha+\|v_1\|_\alpha\|w_c\|_0+\|w_c\|_0\|w_c\|_\alpha)\\
&\stackrel{\eqref{e:w_c_est}}{\leq} C \frac{\mu}{\lambda^{1-\alpha}} ( \|v_1\|_\alpha + \|w_c\|_\alpha)\\
&\leq C \frac{\mu}{\lambda^{1-\alpha}} (\|v\|_\alpha + \|w_c\|_\alpha + \|w_o\|_\alpha)\\
&\leq C\frac{\mu}{\lambda^{1-\alpha}} \left(C + C \frac{\mu}{\lambda^{1-\alpha}} + C \lambda^\alpha\right)\, .
\end{split}
\end{equation*}
Recall that $1\leq \mu\leq \lambda$ and hence $1\leq \frac{\mu}{\lambda^{1-\alpha}}\leq \lambda^\alpha$.
Thus we conclude
\[
\|v_1\otimes w_c+w_c\otimes v_1-w_c\otimes w_c\|_\alpha\leq C \frac{\mu}{\lambda^{1-2\alpha}}\, .
\]
\eqref{e:error2} follows from the latter inequality and the Schauder estimate \eqref{e:Schauder_Rdiv}. 
\end{proof}

\bigskip

\begin{lemma}[Estimate on the error III]
\begin{equation}\label{e:error3}
\|\RR\left(\div(v\otimes w_o)\right)\|_{\alpha}\leq C\frac{\mu^2}{\lambda^{1-\alpha}}.
\end{equation}
\end{lemma}

\begin{proof}
Since $B_k\cdot k=0$, we can write
\begin{equation*}
\begin{split}
\div(v\otimes w_o)&=w_o\cdot\nabla v+(\div w_o)v\\
&=\sum_{|k|=\lambda_0}\left[a_k(B_k\cdot \nabla)v+v(B_k\cdot\nabla a_k)\right]e^{i\lambda k\cdot x}
\end{split}
\end{equation*}
The claim follows from Corollary \ref{c:schauder} with $m=1$.
\end{proof}

\section{Conclusion: Proof of Proposition \ref{p:iterate}}\label{s:conclusion}

In this section we collect the estimates from the preceding sections. For simplicity 
we set
$$
\mu=\lambda^{\beta}.
$$
It should be noted, however, that due to the requirement \eqref{e:integrality} we can only ensure $\mu\sim\lambda^{\beta}$ for large $\lambda$. 

We claim that for an appropriate choice of $\alpha$ and $\beta$, the estimates
\eqref{e:energy_est},\eqref{e:Reynolds_est}, \eqref{e:C^0_est} and \eqref{e:pressure_est}
will be satisfied for sufficiently large $\lambda$.
First of all recall that, by the choice of $M$ we have
\begin{equation}\label{e:est_wo_recalled}
\|w_o\|_{0}\leq \frac{\sqrt{M\delta}}{2}\, 
\end{equation}
(cp. with \eqref{e:est_wo}) and $M>1$.
Therefore \eqref{e:C^0_est} follows from the estimate $\|w_c\|_\alpha \leq C \mu \lambda^{\alpha-1}$ (cp.
with \eqref{e:w_c_est}) if, for instance, 
we can prescribe
$$
C\frac{\mu}{\lambda^{1-\alpha}}=C\lambda^{\alpha+\beta-1}\leq\frac{\sqrt{\delta}}{2}.
$$
On the other hand \eqref{e:pressure_est} follows easily from \eqref{e:est_wo_recalled} and the identity $p_1-p=-\frac{1}{2}|w_o|^2$.

Also, from \eqref{e:energy} it follows that \eqref{e:energy_est} is satisfied provided
$$
C\lambda^{\alpha+\beta-1}\leq \frac{1}{8}\delta \min_{t\in [0,1]}e(t).
$$
Finally, \eqref{e:transport},\eqref{e:oscillation} as well as the estimates on the error \eqref{e:error1}-\eqref{e:error3} imply that
$$
\|\mathring{R}_1\|_{\alpha}\leq C\left(\lambda^{\alpha-\beta}+\lambda^{\alpha+2\beta-1}+\lambda^{2\alpha+\beta-1}\right).
$$
Therefore, any choice of $\alpha,\beta$ such that
\begin{equation}\label{e:alphabeta}
\alpha<\beta,\quad \alpha+2\beta<1
\end{equation}
will ensure that \eqref{e:energy_est},\eqref{e:Reynolds_est}, \eqref{e:C^0_est} and \eqref{e:pressure_est}
will be valid for sufficiently large $\lambda$.
This completes the proof.
As a side remark observe that \eqref{e:alphabeta} requires $\alpha<\frac{1}{3}$.

%\bibliographystyle{acm}
%\bibliography{eulerbib}

\end{document}